\documentclass[a4paper,12pt]{article}
\usepackage[T2A]{fontenc}
\usepackage[english]{babel}

\usepackage[usenames]{color}
\usepackage{colortbl}
\usepackage{amsfonts}
\usepackage{amsmath}
\usepackage{pb-diagram}

\usepackage{amssymb}
\usepackage{amsthm}
\usepackage{amstext}

   \theoremstyle{plain}
   \newtheorem{thm}{Theorem}[section]
   
   \newtheorem{lem}[thm]{Lemma}
   
   \theoremstyle{definition}

   \theoremstyle{remark}

\title{Amalgamated free products of commutative $C^*$-algebras are residually finitedimensional}

\author{A. Korchagin}

\date{}

\begin{document}

\maketitle

\begin{abstract}
We prove that an amalgamated free product of separable commutative $C^*$-algebras is residually finitedimensional.

\end{abstract}


\section{Introduction}

Recall that a $C^*$-algebra is residually finitedimensional (RFD) if it separable and admits an embedding in a direct product of finitedimensional $C^*$-algebras. In other terms, a $C^*$-algebra $A$ is RFD if $\|a\|=\sup\|\varphi(a)\|$ for any $a\in A$, where the supremum is taken over all finitedimensional representations $\varphi$ of $A$.
In this paper we prove the RFD property for amalgamated products of commutative $C^*$-algebra.

At the end of the article we demonstrate an application of this theorem to one interesting example.


Recall that if $\phi_A:C\to A$, $\phi_B:C\to B$ are unital $*$-homomorphisms of unital $C^*$-algebras then their amalgamated free product (or simply amalgam) $A\underset{C}{\star}B$ is a $C^*$-algebra with the following properties:
\begin{enumerate}
\item
There exist $*$-homomorphisms $\varphi_A:A\to A\underset{C}{\star}B$ and $\varphi_B:B\to A\underset{C}{\star}B$ such that the square
$$
\begin{diagram}
\node{C}
  \arrow{e,l}{\phi_A}
  \arrow{s,l}{\phi_B}
  \node{A}
  \arrow{s,l}{\varphi_A}\\
\node{B}
  \arrow{e,l}{\varphi_B}
  \node{A\underset{C}{\star}B}
\end{diagram}
$$
commutes;
\item
For any $C^*$-algebra $D$ and for any commutative square
$$
\begin{diagram}
\node{C}
  \arrow{e,l}{\phi_A}
  \arrow{s,l}{\phi_B}
  \node{A}
  \arrow{s,l}{\tilde{\varphi}_A}\\
\node{B}
  \arrow{e,l}{\tilde{\varphi}_B}
  \node{D}
\end{diagram}
$$
there is a unique $*$-homomorphism $\Phi:A\underset{C}{\star}B\to D$ such that $\Phi\circ\varphi_A=\tilde{\varphi}_A$ and $\Phi\circ\varphi_B=\tilde{\varphi}_B$.
\end{enumerate}
Such $C^*$-algebra exists and is unique up to isomorphism (for information see [Ped, section 2.3]).

Some examples:
\begin{enumerate}
\item
$C(\mathbb{T})\underset{\mathbb{C}}{\star}C(\mathbb{T})\cong C^*(\mathbb{F}_2)$, where $C(\mathbb{T})$ is the algebra of continuous functions over the circle $\mathbb{T}$ and $C^*(\mathbb{F}_2)$ is the full group $C^*$-algebra of a free group on two generators.

\item $\mathbb{C}^2\underset{\mathbb{C}}{\star}\mathbb{C}^2\cong\mathbb{C}^*(p,q)$ is the universal $C^*$-algebra generated by two selfadjoint projections without any additional relations [Ped, remark 5.6].

\item $A\underset{A}{\star}A\cong A$.

\item $A\underset{\mathbb{C}}{\star}\mathbb{C}\cong A$.

\item $C^*(G)\underset{C^*(H)}{\star}C^*(T)\cong C^*(G\underset{H}{\star}T)$, where $G\underset{H}{\star}T$ is the amalgamated free product of groups [DHM, Theorem 4].
\end{enumerate}

Unlike the above examples, most of the amalgams have no explicit description and can be described only by their universality property.

\vskip 22pt

If a separable $C^*$-algebra can be embedded in a direct product (one can make it countable) of matrix algebras,
$$A\hookrightarrow\prod_k M_{n_k}(\mathbb{C})$$
then we say, that $A$ has the RFD property or simply that $A$ is an RFD algebra.

Recall that every RFD algebra has a trace, e.g. in the countable case it can be defined by the formula
$\tau=\sum_k \frac{\tau_k}{2^k}$,
where $\tau_k$ is the normalized matrix trace on $M_{n_k}(\mathbb{C})$. Non-existence of a trace often helps to disprove the RFD property, but in our case it cannot help.

Here are some basic examples:
\begin{enumerate}
\item Finite-dimensional $C^*$-algebras are all RFD.

\item If $A$ and $B$ are RFD algebras then $A\underset{\mathbb{C}}{\star}B$ is an RFD algebra [EL].

\item If $A$ and $B$ are RFD algebras then $A\oplus B$ is RDF.

\item For any compact Hausdorff space $X$, the algebra $C(X)$ of continuous functions over $X$ is RFD.

\item The $C^*$-algebra $\mathbb{K}(\mathbb{H})$ of all compact operators over a separable Hilbert space $\mathbb{H}$ is not an RFD algebra as it has no trace.

\item $M_2(\mathbb{C})\underset{\mathbb{C}^2}{\star}M_3(\mathbb{C})$, where amalgamation is constructed by embeddings $\mathbb{C}^2\hookrightarrow M_2(\mathbb{C})\; :(x,y)\mapsto diag(x,y)$ and $\mathbb{C}^2\hookrightarrow M_3(\mathbb{C})\; :(x,y)\mapsto diag(x,y,0)$, is not an RFD algebra because it has no trace [QiJu, Example 2.1].
\end{enumerate}

Here is the main result of the paper:

\begin{thm}
Let $A,B \supseteq C$  be separable commutative unital $C^*$-algebras. Then $A\underset{C}{\star} B$ is an RFD algebra.

\end{thm}

\section{Lemmas and Proof of the Theorem}

\begin{lem}[abundance of invariant subspaces]\label{invariant}
Let
$$
\varphi:\mathcal{M}=C(X)\underset{C(Z)}{\star}C(Y)\longrightarrow\mathbb{B}(\mathbb{H})
$$
is a unital representation on some Hilbert space $\mathbb H$.
Then, every Borel set $\mu\subseteq Z$ corresponds an invariant subspace $\mathbb{H}_{\mu}\subseteq\mathbb{H}$ with the properties:
\begin{enumerate}
\item $\mathbb{H}_{\mu}\bot\mathbb{H}_{\nu}$, whenever $\mu\cap\nu=\emptyset$;
\item if $Z=\coprod_{k=1}^{N}\mu_k$ then $\mathbb{H}=\bigoplus_{k=1}^{N}\mathbb{H}_{\mu_k}$;
\item $\mathbb{H}_\mu\subseteq\mathbb{H}_\nu$ whenever $\mu\subseteq\nu$.
\end{enumerate}
\end{lem}
\begin{proof}
In [Pir, $\S$7.3$-\S$7.4] for every Hausdorff compact $Z$ and representation $\varphi$ commutative triangle is constructed:
$$
\begin{diagram}
  \node{C(Z)}
    \arrow{e,l}{\varphi}
    \arrow{s,l}{p}
    \node{\mathbb{B}(\mathbb{H})}\\
  \node{\mathcal{B}(Z)}
    \arrow{ne,r}{\widehat{\varphi}}
\end{diagram}
$$
where $\mathcal{B}(Z)$ is algebra of all bounded complex-valued Borel functions, $\widehat\varphi$ is continuous with respect to $\omega_0 - WOT$-topology, where
$WOT$ - is the weak topology on $\mathbb{B}(\mathbb{H})$ and
$\omega_0$ is the weak-measure topology, which is defined by the collection of semi-norms $\|f\|_\mu=|\int_Z f d\mu|$ parametrized by regular Borel measure $\mu$ with bounded variation.

[Remark: $\mathcal{B}(Z)$ is involutive, but not even a Banach algebra]

It is also known that for bounded sequences $\omega_0$-topology is equivalent to the point-wise convergerce topology.

Set
$$\mathbb{H}_\mu=Im\widehat\varphi(\chi_\mu)$$
where $\chi_\mu$ is a characteristic function of $\mu\subseteq Z$. Properties 1)-3) easily follow from similar properties of characteristic functions. To prove that $\mathbb{H}_\mu$ is invariant subspace, let us construct sequence of continuous functions $\{f_k\}$ such that
$$f_k\xrightarrow{\text{point-wise}}\chi_\mu.$$
Then
$$\varphi(f_k)\xrightarrow{\text{WOT}}\widehat\varphi(\chi_\mu).$$
As $f_k\in C(Z)$ lie in the center of the amalgam $\mathcal{M}$ [Thom], so $\varphi(f_k)$ commute with all $\varphi(\mathcal{M})$. Finally, for arbitrary $\omega\in\varphi(\mathcal{M})$ we have
$$\omega \widehat\varphi(\chi_\mu)=\omega\text{lim-}WOT \varphi(f_k)=\text{lim-}WOT \varphi(f_k)\omega=\widehat\varphi(\chi_\mu)\omega.$$
It proves invariance of $\mathbb{H}_\mu$.
\end{proof}


\begin{lem}[topological]\label{topological}
Let $K$ be a metric compact space and $\nu_k\subseteq K$, $k\in\mathbb N$, its compact subsets with the property that for every $n\in\mathbb{N}$ one has $$\nu_{n+1}\subseteq \nu_n.$$
Set $\Gamma=\bigcap_{n\in\mathbb{N}} \nu_n$.
Then for every $\varepsilon>0$, there is $N\in\mathbb N$ such that for every $n>N$
$$\nu_n\subseteq\mathcal{O}_{\varepsilon}(\Gamma)$$
i.e. $\varepsilon$- neighbourhood of $\Gamma$ contains every $\nu_n$ with $n>N$

\end{lem}
\begin{proof}
The proof is an easy exercise.
\end{proof}

%
%

Recall that $\mathcal{M}=C(X)\underset{C(Z)}{\star} C(Y).$

By Gelfand theory for commutative algebras there are natural continuous maps: $p_X:X\to Z$ and $p_Y:Y\to Z$.

For arbitrary $\nu\subseteq Z$ introduce the notation: $\hat\nu=p_X^{-1}(\nu)$ and $\check\nu=p_Y^{-1}(\nu)$.

Let $Z\supseteq\nu_1\supseteq\nu_2\supseteq...$ be compacts with $diam (\nu_n)\to 0$ that implies  $\bigcap_n \nu_n=\bullet$, where $\bullet$ denotes a point. Set
$$\mathcal{M}_n=C(\hat\nu_n)\underset{C(\nu_n)}{\star}C(\check\nu_n)$$ $$\mathcal{M}_\infty=C(\hat\bullet)\underset{C(\bullet)}{\star}C(\check\bullet).$$
Then, for nested compacts, we can construct the chains
$$C(Z)\to C(\nu_1)\to C(\nu_2)\to ... \to C(\bullet)$$
$$C(X)\to C(\hat\nu_1)\to C(\hat\nu_2)\to ... \to C(\hat\bullet)$$
$$C(Y)\to C(\check\nu_1)\to C(\check\nu_2)\to ... \to C(\check\bullet).$$
Gelfand theory can describe subalgebras of commutative algebras
$$C(Z)=\{f\in C(X):\;f|_{p_X^{-1}(*)}=\text{const}\;\forall *\in Z\}.$$
 Due to this characterization, we can lift these homomorphisms to amalgams (homomorphisms are admissible on common subalgebras $C(\nu_n)$ i.e. form necessary commutative triangles):
$$\mathcal{M}\xrightarrow{\alpha_0}\mathcal{M}_1\xrightarrow{\alpha_1}\mathcal{M}_2\xrightarrow{\alpha_2}...\to\mathcal{M}_\infty$$
\begin{lem}[main]\label{main}

One has $$\underrightarrow{\lim}\mathcal{M}_n=\mathcal{M}_\infty.$$
\end{lem}

\begin{proof}
One can check (using nice commutation properties of our chains) that $\alpha_n$ induce a well-defined homomorphism
$$\underrightarrow{\lim}\mathcal{M}_n\xrightarrow{\alpha_\infty}\mathcal{M}_\infty.$$
Let us construct a homomorphism $\gamma_\infty:\mathcal{M}_\infty\to \underrightarrow{\lim}\mathcal{M}_n$.
As $C(\hat\bullet)$ and $C(\check\bullet)$ generate $\mathcal{M}_\infty$ [Thom], so $\gamma_\infty$ could be define only on them.

For $\omega\in C(\hat\bullet)$ (similarly for $C(\check\bullet)$) set
$$\gamma_\infty(\omega)=\{\Omega|_{\nu_1},\Omega|_{\nu_2},...\}$$
where $\Omega$ is an arbitrary extension of $\omega$ by Titze-Urysohn theorem.

The map $\gamma_\infty$ is well-defined, as for another extension $\Sigma$ of $\omega$, we have
$$(\Omega-\Sigma)|_{\hat\bullet}=0.$$
 As $\Omega-\Sigma$ is uniformly continuous on compact $Z$, so by Lemma \ref{topological} for any $\varepsilon>0$ there is $n\in\mathbb{N}$ such that
$$\|(\Omega-\Sigma)|_{\hat\nu_n}\|\leq \varepsilon.$$
This means that in $\underrightarrow{\lim}\mathcal{M}_n$ one has the equality $$\{\Omega|_{\nu_1},\Omega|_{\nu_2},...\}=\{\Sigma|_{\nu_1},\Sigma|_{\nu_2},...\}.$$
It easy to check that $\gamma_\infty$ is a homomorphism (\`{a} la product of admissible sequences is admissible for product...). As $\gamma_\infty$ is unital on $C(\hat\bullet)$ (and on its twin $C(\check\bullet)$) and $C(\bullet)=\mathbb{C}$, so we can extend it to $\mathcal{M}_\infty$. As $\gamma_\infty\circ\varphi_\infty=\text{id}$ and $\varphi_\infty$ is surjective, so $\varphi_\infty$ is an isomorphism.
\end{proof}
We remark that reader can find this lemma in more general terms in [Ped, Proposition 4.12]

Let $\{\mu_1,...,\mu_N\}$ be a finite covering of $Z$ by compact sets. Set $$P_{\mu_n}\mathcal{M}=C(\hat\mu_n)\underset{C(\mu_n)}{\star}C(\check\mu_n).$$
(the meaning of this notation will be come clear later)

As $\mu_n\subseteq Z$, so we can construct a homomorphism $\gamma_n$
$$\gamma_n:\mathcal{M}\to P_{\mu_n}\mathcal{M}$$
which on $x\in C(X)$ is defined by formula
$$x\mapsto x|_{\mu_n},$$
has similar definition on $C(Y)$ and extends to $\mathcal{M}$.

\begin{lem}[decomposition of an amalgam]\label{IV}

The map
$$\gamma=\prod_{m=1}^N\gamma_m:\mathcal{M}\to \prod_{m=1}^NP_{\mu_m}\mathcal{M}$$
is injective.
\end{lem}

\begin{proof}
To prove we have to check the equality $\|\omega\|=\max_m\|\gamma_m(\omega)\|$ for arbitrary $\omega\in\mathcal{M}$. By elementary properties of *-homomorphisms we have
$$\max_m\|\gamma_m(\omega)\|\leq\|\omega\|$$
By Gelfand-Naimark  theorem we can construct a faithful representation
$$\mathcal{M}\hookrightarrow\mathbb{B}(\mathcal{H})$$
By Lemma \ref{invariant} we can restrict representation on $\mathbb{H}_m$. Let $\varphi_m=P_{\mathbb{H}_{\mu_m}}\varphi$
$$\mathcal{M}\xrightarrow{\varphi_m}\mathbb{B}(\mathbb{H}_m).$$
As $Z=\bigcup_{m=1}^N\mu_m$ we obtain by Lemma \ref{invariant} that
$$\mathbb{H}\subseteq\mathbb{H}_{\mu_1}+...+\mathbb{H}_{\mu_N}$$
$$\|\omega\|\leq\max_m\|\varphi_m(\omega)\|$$
because we can easily find disjoint Borel sets $\tilde\mu_m$ such that $\tilde\mu_m\subseteq\mu_m$ and $Z=\coprod_{m=1}^N\tilde\mu_m$. As $\mathbb{H}_{\tilde\mu_m}$ are orthogonal by Lemma \ref{invariant}, so by properties of block-diagonal operators one has $\|\varphi(\omega)\|=\max_m\|P_{\mathbb{H}_{\tilde\mu_m}}\varphi(\omega)\|$
As $\tilde\mu_m\subseteq\mu_m$ we have
$$\max_m\|P_{\mathbb{H}_{\tilde\mu_m}}\varphi(\omega)\|\leq\max_m \|\varphi_m(\omega)\|.$$
There is an isomorphism:
$$\varphi_m(\mathcal{M})\cong P_{\mu_m}\mathcal{M}$$
which is defined for $x\in C(\hat\mu_m)\subseteq P_{\mu_m}\mathcal{M}$ by formula
$$\delta:x\mapsto P_{\mathbb{H}_{\mu_m}}\varphi(\hat{x}).$$
where $\hat{x}\in C(X)\subseteq\mathcal{B}(X)$ is an arbitrary continuous extension of $x$ to X. It is well-defined as $P_{\mathbb{H}_{\mu_m}}\varphi(\hat{x})=\widehat{\varphi}(\hat{x}\chi_{\mu_m})$, has similar definition for $y\in C(\check\mu_m)$ and admits an extension to $P_{\mu_m}\mathcal{M}$ (about $\widehat{\varphi}$ see [Pir] and Lemma \ref{invariant}).
Define for $P_{\mathbb{H}_{\mu_m}}\varphi(x)\in \varphi_m(C(X))$ map
$$\Delta: P_{\mathbb{H}_{\mu_m}}\varphi(x)\mapsto x|_{\mu_m}.$$
 Representation $\varphi$ is injective then $P_{\mathbb{H}_{\mu_m}}\varphi(x-y)=0$ implies $x|_{\mu_m}=y|_{\mu_m}$; so $\Delta$ is a well-defined homomorphism and admits extension to $\varphi_m(\mathcal{M})$. As $\Delta$ is surjective and $\Delta\circ\delta=id$, so $\Delta$ is isomorphism. Finally, we have $\Delta\circ\varphi_m(\omega)=\gamma_m(\omega)$ and $\|\omega\|\leq\max_m\|\gamma_m(\omega)\|.$
\end{proof}

\begin{lem}[RFD norms]\label{V}

Let $A$, $B$ be $C^*$-algebras, $a\in A$. Set
$$\|a\|_{RFD}=\sup_{{\phi\text{  finite}}\atop{{\text{dimemsional}}\atop{\text{representation}}}}\|\phi(a)\|.$$
If $\varphi:A\rightarrow B$ is a $*$-homomorphism then $\|\varphi(a)\|_{RFD}\le\|a\|_{RFD}$.

\end{lem}
\begin{proof}
Obvious.
\end{proof}

\begin{proof}[Proof (of the theorem)]

By Gelfand-Naimark theorem $A=C(X),B=C(Y),C=C(Z)$, where $X$, $Y$ and $Z$ are metric compacts.
Suppose that $\mathcal{M}$ is not an RFD algebra. Then for some $\varepsilon>0$ and some $0\neq\omega\in\mathcal{M}$ we have
$$\|\omega\|\geq(1+\varepsilon)\|\omega\|_{RFD}.$$
Let $\{\mu_1,...,\mu_N\}$ is a finite covering $Z$ by compact sets such that $diam(\mu_n)\leq\frac{diam(Z)}{2}$.

By Lemmas \ref{IV}, \ref{V}, we have
$$\|\omega\|=\max_n\|\gamma_n(\omega)\|,$$
$$\|\omega\|_{RFD}\geq\max_n\|\gamma_n(\omega)\|_{RFD}.$$
Then for some $n_0$, we have
$$\|\omega\|=\|\gamma_{n_0}(\omega)\|\geq(1+\varepsilon)\|\gamma_{n_0}(\omega)\|_{RFD}$$
Let $\nu_1=\mu_{n_0}$, $\alpha_0=\gamma_{n_0}$ and $\mathcal{M}_1=C(\hat\nu_1)\underset{C(\nu_1)}{\star}C(\check\nu_1)$. Now, let us apply this decomposition method to $\nu_1$ in place of $Z$, namely, let us find compacts $\mu_k$ (with corresponding homomorphisms) such that $\nu_1=\bigcup_k \mu_k$ and $diam(\mu_k)\le\frac{diam(\nu_1)}{2}$. Now we can find $n_1$ such that
$$\|\omega\|=\|(\gamma_{n_1}\circ\alpha_0)(\omega)\|\geq(1+\varepsilon)\|(\gamma_{n_1}\circ\alpha_0)(\omega)\|_{RFD}.$$
Let $\nu_2=\mu_{n_1}$, $\alpha_1=\gamma_{n_1}$ and $\mathcal{M}_2=C(\hat\nu_2)\underset{C(\nu_2)}{\star} C(\check\nu_2)$. Then let us apply this decomposition method to $\nu_2$, etc.

Inductively we have
$$\mathcal{M}\xrightarrow{\alpha_0}\mathcal{M}_1\xrightarrow{\alpha_1}\mathcal{M}_2\xrightarrow{\alpha_2}...$$
Let $\Phi_n=\alpha_n\circ ...\circ\alpha_0$. Apply Lemma \ref{main} to this sequence. As $diam(\nu_n)\rightarrow 0$, so $\underrightarrow{\lim}\mathcal{M}_n\cong\mathcal{M}_\infty=C(\hat\bullet)\underset{C(\bullet)}{\star}C(\check\bullet)$. For $\Phi_\infty(\omega)=\{\Phi_0(\omega),\Phi_1(\omega),...\}\in\underrightarrow{\lim}\mathcal{M}_n$ we have
$$0\neq\|\omega\|=\|\Phi_\infty(\omega)\|=\liminf_n\|\Phi_n(\omega)\|\geq(1+\varepsilon)\liminf_n\|\Phi_n(\omega)\|_{RFD} \geq(1+\varepsilon)\|\Phi_\infty(\omega)\|_{RFD}.$$
The last inequality follows from the existence of the canonical $*$-homomorphism $\mathcal{M}_n\longrightarrow\mathcal{M}_\infty:\;\Phi_n(\omega)\mapsto\Phi_\infty(\omega)$ and from Lemma \ref{V}. As $C(\bullet)$ is one-dimensional, so $\mathcal{M}_\infty$ is RFD (as a free product of RFD algebras [EL]). But this contradicts our inequality $\|\Phi_\infty(\omega)\|\geq(1+\varepsilon)\|\Phi_\infty(\omega)\|_{RFD}$, so our supposition was wrong.
\end{proof}

\section{Concluding remarks}

{\bf An example (of application of the theorem)}

In [ManCho], the authors consider the universal $C^*$-algebras $\mathcal{A}_\lambda=\mathbb{C}^*(U,V)$ generated by two unitaries $U$ and $V$ with the property
$$\|Re(U+V)\|\le\lambda, \qquad \lambda\in[0,2].$$
They show that $$\mathcal{A}_0\cong C(\mathbb{T})\underset{C(\mathbb{I})}{\star} C(\mathbb{T})$$ where $C(\mathbb{T})$ is the algebra of continuous functions over the unit circle $\mathbb{T}$ and $C(\mathbb{I})$ is the algebra of continuous functions over the segment $[-1,1]$. We consider $\mathbb T$ and $\mathbb I$ as subsets on the complex plane. The map $\mathbb{T}\longrightarrow\mathbb{I},\; z\mapsto Re(z)$, defines $*$-homomorphisms of the algebras in the natural way. RFD property for $\mathcal{A}_0$ follows from our theorem. For $\mathcal{A}_2=\mathbb{C}^*(\mathbb{F}_2)$, RFD property was proved by Choi in [Choi]. For $\mathcal{A}_\lambda$, where $\lambda\in (0,2)$, RFD property is an open question.

{\bf Remarks}

Construction of commutative triangle:
$$\begin{diagram}
  \node{C(X)}
    \arrow{s}
    \arrow{e,l}{\varphi}
    \node{\mathbb{B}(\mathbb{H})}\\
  \node{\mathcal{B}(X)}
    \arrow{ne,r}{\widehat\varphi}
\end{diagram}$$
reader can find in [Pir]. In [Thom] we can find a proof of the fact that finite sums of finite products of elements of algebras, which define an amalgam, are dense in it. V.M.Manuilov was the first, who considered algebras $\mathcal{A}_\lambda$. Paper [ChMan] was motivation for the main theorem. Paper [QiJu] is very interesting in this theme.

We also remark that using this method one can prove RFD property for an amalgamated product of many commutative algebras, so is for algebras of the form: $$A_1\underset{C}{\star}A_2\underset{C}{\star}...\underset{C}{\star}A_n$$
where $C$ is a subalgebra of commutative algebras $A_1$, $A_2$...$A_n$: our consideration about projections, embedding, block-diagonal operators and decompositions does not depend on quantity of amalgamated commutative algebras.

{\bf Acknowledgment}
The author is greatful to V.M.Manuilov for grand help and support, the Bogolubov Laboratory of Geometrical Methods Mathematical Physics for comfortable atmosphere and my Muse Mary for inspiration.

\end{document}